\documentclass[runningheads]{llncs}

\usepackage{ifthen,
amsmath,
amssymb,latexsym,color,epsfig,
bm,verbatim,
}

\usepackage{hyperref}
\usepackage{newtype}

\def\rlcr{\mathop{\overline{\rm lcr}}\nolimits}

\def\itemi{\item [$(i)$]}
\def\itemii{\item [$(ii)$]}

\def\itemiii{\item [$(iii)$]}

\def\RAC{\ensuremath{\mathop{\rm RAC}\nolimits}}

\newcommand{\Tvx}{\ensuremath{\top}}
\newcommand{\Pvx}{\ensuremath{\times}}

\newcommand{\RN}{\mathbb{R}}

\newtype{\class}{\mathbf}
\class{P}
\class{NP}
\class{AM}
\class{LOGNP}
\class{coNP}
\class{NTIME}
\class{PSPACE}
\class{EXP}
\class{EXPSPACE}
\class{NEXP}
\class{RP}
\class{coRP}
\class{TA}
\class{RA}
\class{PP}
\class{PH}
\class{CH}
\class{FIXP}
\class[FIXPd]{FIXP_d}
\class[sharpP]{\#P}

\newcommand{\NPR}{\ensuremath{\bm{\exists \RN}}}

\begin{document}

\title{\RAC-Drawability is $\exists \mathbb{R}$-complete}

\author{Marcus Schaefer
\orcidID{0000-0001-7005-8599} }

\authorrunning{M. Schaefer}
%
\institute{DePaul University, Chicago, IL 60604, USA
\email{mschaefer@cdm.depaul.edu}\\
\url{http://ovid.cs.depaul.edu}
}

\maketitle

\begin{abstract}
 A {\em \RAC-drawing} of a graph is a straight-line drawing in which every crossing occurs at a right-angle. We show that deciding whether a graph has a \RAC-drawing is as hard as the existential theory of the reals, even if we know that every edge is involved in at most ten crossings and even if the drawing is specified up to isomorphism.
\end{abstract}



\section{Introduction}

It is generally admitted that if we cannot avoid crossings in drawings, then it is advantageous to draw the crossings with large angles. This simplifies reading the drawing as edges become easier to follow individually. In a 2009 paper, Didimo, Eades and Liotta~\cite{DEL11} formalized this idea by introducing \RAC-drawings of graphs, in which only the largest possible angle, the right-angle, is allowed. In other words, a drawing of a graph is a {\em \RAC-drawing} if all crossings in the drawing occur at right angles.

\RAC-drawings have become an increasingly popular subject in the graph drawing literature, see, for example, a recent survey by Didimo~\cite{D20} summarizing our knowledge. We are specifically interested in the computational complexity of the recognition problem.

Argyriou, Bekos, and Symvonis~\cite{ABS12} showed early on that it is \NP-hard to recognize whether a graph has a \RAC-drawing. Why \NP-hard, and not
\NP-complete? The issue is that a realization of a \RAC-drawing may require
real coordinates, and, a priori, we do not have any bounds on the precision and we do not
even know whether the graph can be realized on a grid. Bieker~\cite{B20} showed that the problem lies in \NPR, the complexity class associated with deciding the truth of the existential theory of the reals.\footnote{For a thorough introduction to the existential theory of the reals, check out~\cite{M14}. For a quick intro, the
Wikipedia page~\cite{WETR21} will serve.} The exact complexity remained open (as mentioned, for example, in~\cite[p. 4:11/12]{DLM19}). 

Also open, not even known to be \NP-hard, was the complexity of the {\em fixed embedding} variant of \RAC-drawability, in which we are given a drawing of the graph and have to decide whether the graph has a \RAC-drawing isomorphic to the given drawing.

\begin{theorem}\label{thm:RACNPR}
  Testing whether a graph (with or without fixed embedding) has a \RAC-drawing is \NPR-complete, even if each edge has at most ten crossings.
\end{theorem}

\NPR-hardness implies \NP-hardness~\cite{S91}, so the fixed embedding variant is \NP-hard. What does \NPR-hardness add that \NP-hardness does not give us? Perhaps nothing, since
it is possible (if considered unlikely) that $\NP = \NPR$. Nevertheless, our \NPR-hardness
reduction shows that \RAC-drawings can require arbitrarily complex algebraic
integers in any realization. And even \RAC-drawings that can be realized on a grid,
may require double-exponential precision. We discuss these results on area and precision in Section~\ref{sec:PaA}.

\subsection{More on \RAC-Drawings}

As far as we know the double-exponential lower bound on the area of a
RAC-drawing is new, but there have been (single) exponential lower bounds in
constrained settings, e.g. for upward \RAC-drawings~\cite{ACDFKS11},
RAC-drawings in which a given horizontal order of the vertices must be realized, and
for $1$-plane \RAC-drawings~\cite{BDEKLM16}, drawings in the plane with at most one crossing
per edge.

If we allow bends along edges, the situation changes dramatically. A
$\RAC_k$ drawing of a graph is a \RAC-drawing in which every edge
has at most $k$ bends. \RAC-drawings are just the $\RAC_0$-drawings.

Every graph has a $\RAC_3$-drawing~\cite{DEL11}, but not necessarily
a $\RAC_2$-drawing: any graph with a $\RAC_2$-drawing has at most linearly many edges~\cite{AFKMT12}. 
The complexity of recognizing graphs with $\RAC_1$- and $\RAC_2$-drawings remains tantalizingly open~\cite[Problem 6]{DLM19}.

There are polynomial upper bounds on the area of $\RAC_k$ drawings for $k \geq 3$~\cite{FK20}, but no bounds seem to be known for $k = 1, 2$. \NPR-hardness of these cases would likely imply
double-exponential lower bounds.

\subsection{Overview}

Our proof that the \RAC-drawability problem is \NPR-complete consists of a sequence of reductions. There is no convenient graph drawing problem to start with, so in Section~\ref{sec:ETR} we present an \NPR-complete algebraic problem. 

The main idea then is to enrich the \RAC-drawing model with a special feature: vertices with a specific type of angle constraints. The proof of Theorem~\ref{thm:RACNPR} then breaks into two major parts. The first part shows that this additional feature leads to an \NPR-complete problem (even for crossing-free drawings), since it is powerful enough to encode the algebraic problem, see Section~\ref{sec:RACAC}. The second part shows how to simulate the special feature within the \RAC-model. This second part had to be relegated to an appendix, Section~\ref{sec:RACNPR}, to meet the length constraints.
In Section~\ref{sec:PaA} we discuss issues of precision, area, and universality, before closing with a short section of open questions.

\section{The Existential Theory of the Reals}\label{sec:ETR}

The {\em existential theory of the reals} is the set of existentially quantified, true statements over the real numbers. The corresponding complexity class is \NPR. (Defined just like \NP\ is defined from Boolean formula satisfiability, though there also is a machine model~\cite{EHM20}.) A problem is {\em \NPR-hard} if every
problem in \NPR\ reduces to it; it is {\em \NPR-complete}, if it is \NPR-hard and lies in \NPR.

\NPR\ captures the complexity of many natural problems in graph drawing, particularly if real coordinates are involved. Bieker's thesis~\cite{B20} surveys many of the relevant graph drawing results.
Some recent problems shown \NPR-complete relevant to graph drawing include: visibility graphs of triangulated irregular networks~\cite{BOZ21}, the local rectilinear crossing number~\cite{S21}, simultaneous geometric embeddings of paths~\cite{S21b}, and covering polygons by triangles~\cite{A21}.\footnote{These results are from 2021. The Wikipedia page mentioned earlier~\cite{WETR21} is host to a growing list of complete problems, many from the areas of graph drawing and computational geometry.}

For our reduction we will be working with an \NPR-complete problem, which, as is often the case, is tailor-made for the situation we find ourselves in. The proof of the theorem can be found in Section~\ref{sec:ETReqs}; it combines ideas from Mn\"{e}v~\cite{M88}, Shor~\cite{S91}, and Richter-Gebert~\cite{RG95}.

\begin{theorem}\label{thm:ETReqs}
   The following problem is \NPR-complete: Given equations of the form
   $x_i = 2$, $x_i = x_j$, $x_i = x_j+x_k$, and 
   $x_i = x_j \cdot x_k$ for 
   variables $x_1, \ldots, x_n$,
   decide whether the equations have a solution with $x_i > 1$ for all $i \in [n]$.
\end{theorem}

As we mentioned, $\NP \subseteq \NPR$~\cite{S91}, in particular, \NPR-hard problems are also \NP-hard. On the other hand, $\NPR \subseteq \PSPACE$~\cite{C88}, so \NPR-complete problems are solvable in polynomial space (and, therefore, exponential time). 

\section{\RAC-Drawings with Angle Constraints}\label{sec:RACAC}

We introduce two types of special vertices that come with angle and rotation constraints. Recall that the rotation at a vertex in a drawing is the clockwise permutation of edges incident to the vertex.

\begin{itemize}
    \item a {\em \Tvx-junction} is a vertex $v$ which is incident to three special edges $e_1, e_2, e_3$ ($v$ may be incident to additional edges). In a straight-line drawing we require that the rotation of the special edges at $v$ is $e_1e_2e_3$, {\em or the reverse}, and there are right angles between $e_1$ and $e_2$ and $e_2$ and $e_3$ at $v$; additional edges at $v$ can occur, at any angle, between $e_1$ and $e_3$ (opposite of $e_2$),
    \item a {\em \Pvx-junction} is a vertex $v$ which is incident to four special edges $e_1, e_2, e_3, e_4$ ($v$ may be incident to additional edges). In a straight-line drawing we require that the rotation of the special edges at $v$ is $e_1e_2e_3e_4$, {\em or the reverse}, and that there are right angles between $e_i$ and $e_{i+1}$, for $1 \leq i \leq 3$; additional edges can occur, at any angle, inside one of the quadrants, e.g. between $e_3$ and $e_4$.
\end{itemize}

Figure~\ref{fig:TvxandPvx} shows these junctions, and how we symbolize
them in drawings.

 \begin{figure}[htb]\centering
 \includegraphics[width=5in]{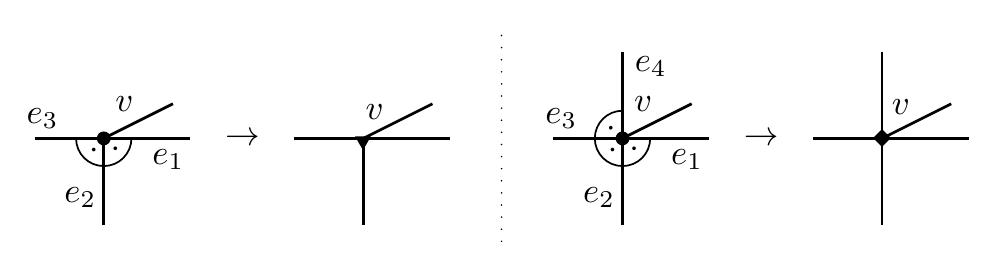}
 \caption{\Tvx- and \Pvx-junctions, and how we draw them in graphs
 using a $\blacktriangledown$ and a $\blacklozenge$. Each of the junctions is shown with one additional edge.}\label{fig:TvxandPvx}
 \end{figure}

 If special edge $e_2$ in a \Tvx-junction ends in a leaf, we can think of the junction as a straight-line, $e_3e_1$, with a vertex on it.

 Two drawings of a graph (with or without junctions) are {\em isomorphic} if there is a  homeomorphism of the plane (which may be orientation-reversing) that maps the graphs to each other.

\begin{theorem}\label{thm:TPvxNPR}
  Testing whether a graph $G$ with \Tvx- and \Pvx-junctions has a \RAC-drawing, is \NPR-complete.
  This remains true even if we are given a crossing-free drawing of $G$ and $G$ either has no
  \RAC-drawing, or it has a crossing-free \RAC-drawing which is isomorphic to the given drawing.
\end{theorem}

We will later see that \Pvx-junctions can be simulated by \Tvx-junctions, so they are not, strictly speaking, necessary, but they do simplify the constructions.

For the promise version note that \Tvx- and \Pvx-junctions are vertices (and not crossings), and that a \RAC-drawing does not have to contain any crossings. The theorem implies that testing whether a graph with \Tvx- and \Pvx-junctions has a crossing-free \RAC-drawing is also \NPR-complete, but we need the stronger promise version for the main proof.

We prepare the proof of Theorem~\ref{thm:TPvxNPR} by constructing gadgets to simulate arithmetic.

\subsection{Gadgets}

We describe the gadgets we use to enforce equations $x_i=2$, $x_i=x_j+x_k$, $x_i = x_j\cdot x_k$ and $x_i= x_j$. Some additional gadgets will be useful.

We start by defining how a drawing of a graph encodes a real number. Given three collinear points (these will be vertices of the graph) labeled $0$, $1$ and $x$, we say $x$ {\em represents} $\pm d(0,x)/d(0,1)$, where $d$ is the Euclidean distance of two points. If $x$ lies on the same side of $0$ as does $1$ (on the common line), we choose the positive value, otherwise we choose the negative value. With this definition we can build our first gadget.

\subsubsection*{Variables}

Figure~\ref{fig:xigadget} shows the gadget we use for variable $x_i$.
The gadget consists of three \Pvx-junctions (labeled $0$, $1$, and $x$ in this order) with the $e_3$-edge of a junction identified with the $e_1$-edge of the next junction. We also add a path connecting the ends $a$, $b$, $c$ of the $e_2$-edges. The $e_4$-edges can then be used to connect to another gadget.

 \begin{figure}[htb]\centering
 \includegraphics[height=0.8in]{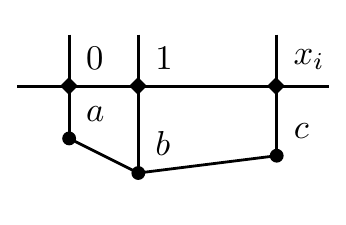}
 \caption{A gadget for variable $x_i$.}\label{fig:xigadget}
 \end{figure}

 In a \RAC-drawing of the $x_i$-gadget, all $e_4$-edges must lie on the same side of the line through $0$, $1$, $x_i$. This is because $ab$ and $bc$ cannot cross the line, since otherwise they would overlap with the $e_4$-edges they are incident to. Thus, the $x_i$-gadget can be used to represent any number $x_i > 1$. By relabeling the junctions, we can also obtain gadgets for variables between $0$ and $1$, and variables less than $0$.

 \subsubsection*{Copying and Moving Information}

Our next gadget will allow us to duplicate information. More precisely, we have points $p_1, \ldots, p_{\ell}$ along a line (in this order).  If two of those points are labeled $0$ and $1$, then all these
points represent numbers. Our gadget allows us to make two copies of these points (along a new common line), that each, by itself, represents the same numbers as the original. See Figure~\ref{fig:copygadget}.

 \begin{figure}[htb]\centering
 \includegraphics{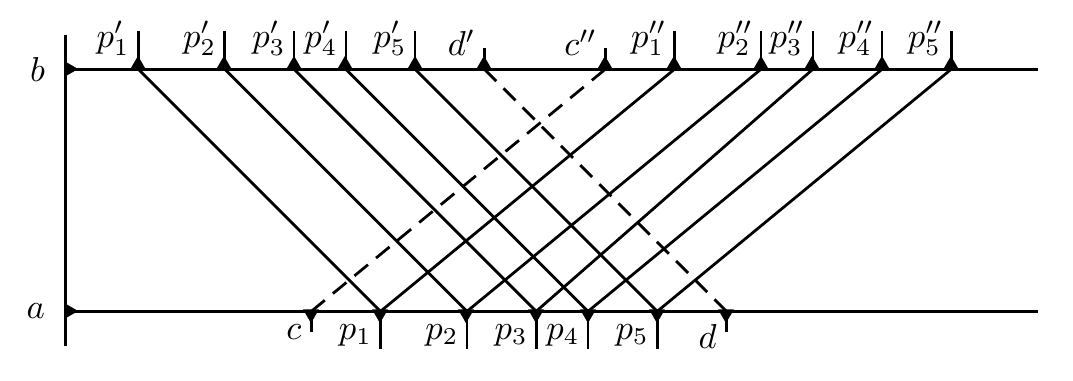}

 \caption{The copy gadget. The dashed edges have to cross orthogonally, forcing the
 $p_ip'_i$ as well as the $p_ip''_i$ edges to be parallel as shown.}\label{fig:copygadget}
 \end{figure}

The two $e_2$-edges incident to the \Tvx-junctions $a$ and $b$ are parallel, since they are both orthogonal to $ab$, and leave $ab$ in the same direction, since otherwise edge $cc''$, for example,
would have to cross $ab$ at a right angle, and overlap $ca$. Because of the relative order of the points on the top and bottom line, edge $cc''$ crosses $dd'$ as well as all edges $p_ip'_i$ at right-angles, so all these edges are parallel. And since every edge $p_jp''_j$ is crossed by $dd'$, these edges are also all parallel. It follows that
$p'_1, \ldots, p'_{\ell}$ and $p''_1, \ldots, p''_{\ell}$ represent the same numbers as $p_1, \ldots, p_{\ell}$.

By repeating the copy gadget, we can make any number of copies of a set of points on a line. There are other applications of the copying gadget, some of which we will see below, but here we can show how to use it to test whether $x_i = x_j$ for two variables $x_i$ and $x_j$ which we know to be larger than $1$. Figure~\ref{fig:eqtestgadget} shows the set-up.

 \begin{figure}[htb]\centering
 \includegraphics{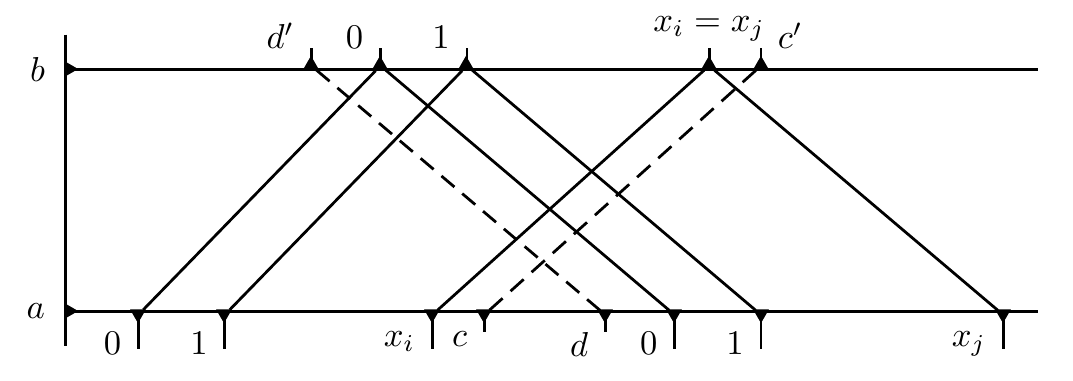}

 \caption{Testing equality $x_i = x_j$ using the copy gadget.}\label{fig:eqtestgadget}
 \end{figure}

We move information around the drawing as parallel lines, with the relative distances of the lines encoding the numbers. However, we may have to change direction, and the gadget shown in Figure~\ref{fig:anglegadget} allows us to create a copy $p'_1, \ldots, p'_{\ell}$ of a set of points $p_1, \ldots, p_{\ell}$ at an angle.

 \begin{figure}[htb]\centering
 \includegraphics{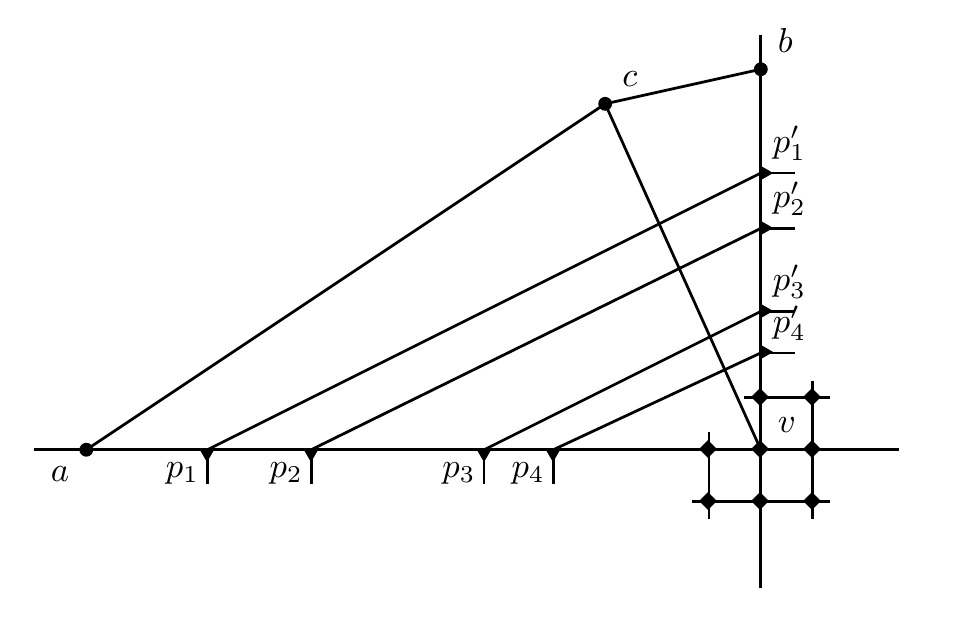}

 \caption{Making a (scaled) copy at a right angle.}\label{fig:anglegadget}
 \end{figure}

 The $p_i$ and $p'_i$ are \Tvx-junctions, and $v$ is a \Pvx-junction. Edges $p_ip'_i$ must lie between the two lines on which the points lie (and on opposite sides
 of the $e_2$-edges of the \Tvx-gadgets for $p_i$ and $p'_i$, since otherwise edges would overlap).
 The \Pvx-junctions surrounding $v$ force $vc$ to also lie between those two
 lines. Then $cv$ must cross each of the edges $p_ip'_i$ which implies that they are parallel,
 so the $p'_i$ have the same relative distances from each other as do
 the $p_i$, and they represent the same numbers.

 By chaining several angling gadget, we can achieve angles of $\pi/2$, $\pi$, $3\pi/2$, and $2\pi$.
 In particular chaining two gadgets, its main axes combined like
 \includegraphics[height=\fontcharht\font`\B]{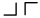}, allows us to continue sending the
 information in the same direction, but at a different
 (arbitrarily chosen) scale. Hence, we can also use the angling gadget to
 rescale information. (The resulting horizontal offset can be compensated for by
 adding two more angling gadgets on top, like
 \includegraphics[height=1.1\fontcharht\font`\B]{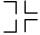}.)


\subsubsection*{Doing Arithmetic I}

Let us start with the number $2$. The gadget shown in Figure~\ref{fig:twogadget} consists of \Pvx-junctions combining four squares with diagonals.

\begin{figure}[htb]\centering
 \includegraphics{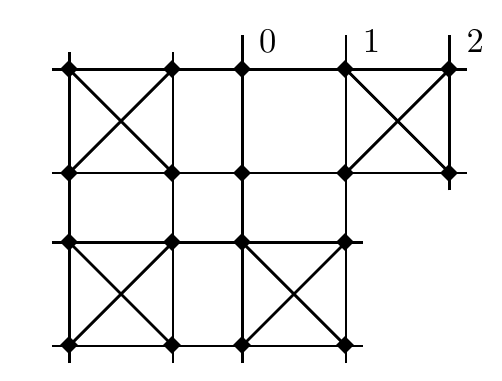}

 \caption{The number $2$ gadget.}\label{fig:twogadget}
\end{figure}

In a \RAC-drawing, the diagonal edges force the sides of each square to have equal length. This forces the length of $01$ and $12$ to be the same.
(Why do we not place two squares with diagonals right next to each other? The reason is that
our \Pvx-junctions only allow additional edges in one quadrant.)

Instead of $0$, $1$, $2$ we can also label the points of the gadget as $-1$, $0$, $1$. This gives us a $-1$ which we can use to build a negation gadget. (While the $x_i$ are all bigger than $1$, we will make use of some additional, temporary, variables that are negative.)

We next build a special negation gadget that, for given points $0$, $1$, and $x$ creates
four points, $-x-1$, $-x$, $0$, and $1$ that correctly represent their labels. See Figure~\ref{fig:neggadget}. On the lower left is a $2$-gadget with points relabeled $-1$, $0$, and
$1$. We added two points to this gadget, labeled $x'$ and $x'+1$. Using the diagonals, the distance between $x'$ and $x'+1$ is the same as the distance between $0$ and $1$ in this gadget. We then use a copy gadget in reverse to combine this information with the $x$-gadget. The resulting five points then correspond to $-1$, $0$, $1$, $x$, and $x+1$, or, more useful for us $-x-1$, $-x$, $-x+1$, $0$ and $1$. Of these we only need four as output, namely $-x-1$, $-x$, $0$ and $1$.

\begin{figure}[htb]\centering
 \includegraphics{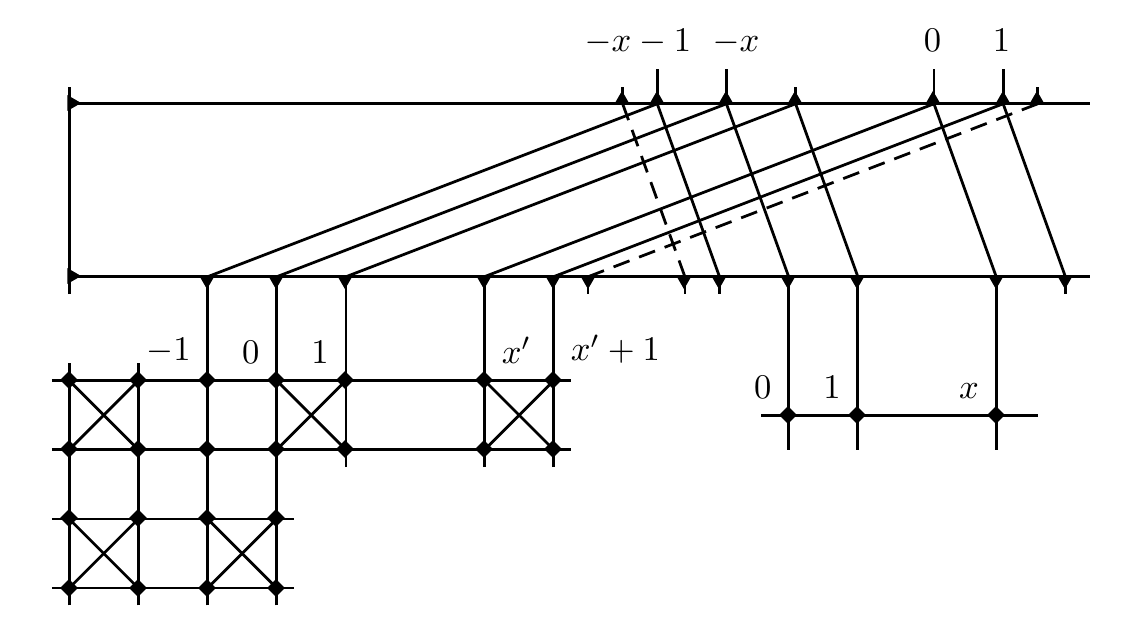}

 \caption{Negating $x$ gadget.}\label{fig:neggadget}
\end{figure}

\subsubsection*{Doing Arithmetic II}

To compute $x_i = x_j+x_k$, we use the negation gadget to create $-x_j-1$ and $-x_j$, and then perform the addition geometrically, as shown in Figure~\ref{fig:addgadget}. The reason for using negation is that this way we do not have to know whether $x_j$ or $x_k$ is the larger one, something that's unavoidable if we place two variables greater than $1$ on the same line. (This idea is due to Richter-Gebert~\cite{RG95}, also see Matou{\v{s}}ek~\cite{M14}.)

\begin{figure}[htb]\centering
 \includegraphics[width=5in]{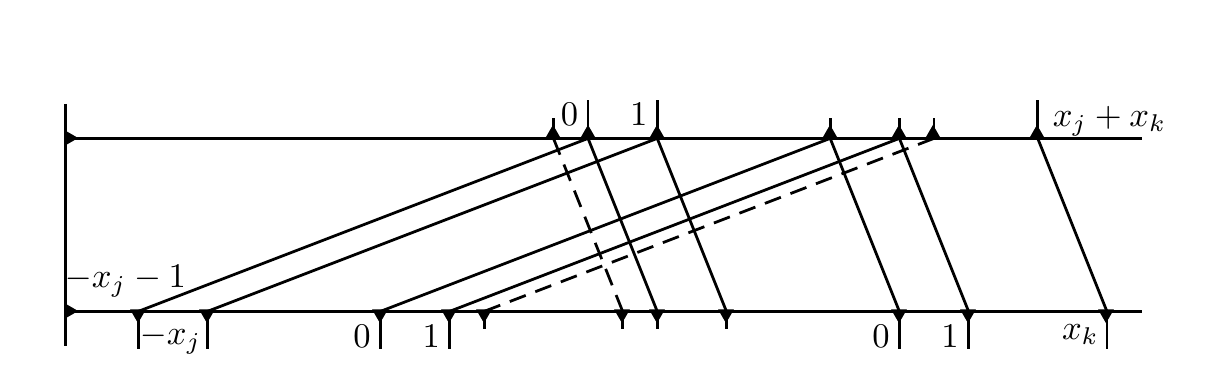}

 \caption{Addition gadget computing $x_i = x_j+x_k$.}\label{fig:addgadget}
\end{figure}

As in the other gadgets, the two dashed edges are orthogonal, and force the other lines
to be parallel. If we ignore those two edges and their endpoints, the remaining points on top represent $-x_j-1$, $-x_j$, $0$, $1$, $x_k$. If we relabel the first and second points as $0$ and $1$, then the fifth point represents $x_j+x_k$, and
we have the addition gadget we needed.

To compute $x_i = x_j\cdot x_k$ we use the same trick mentioned above to resolve the issue that we do not know which of $x_j$ and $x_k$ is bigger, so we cannot place them along the same line. Instead of working with negation, here we work with the reciprocal of $x_j$, which lies (strictly) between $0$ and $1$.  The reciprocal gadget is not much of a gadget, just a relabeling, see Figure~\ref{fig:reciprocalgadget}.

\begin{figure}[htb]\centering
 \includegraphics{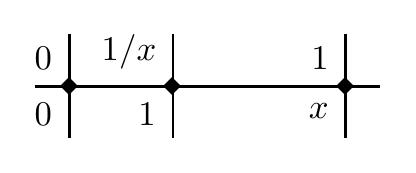}

 \caption{Computing the reciprocal $1/x$ of a variable $x > 1$.}\label{fig:reciprocalgadget}
\end{figure}

This is cheating a bit, since for most other gadgets (except the angling gadget) the distance between $0$ and $1$ did not change, and most gadgets assume they are the same when
processing inputs. So before using the result of the reciprocal gadget, we rescale it, so
that $0$ and $1$ have the standard distance.

With that, the product $x_i = x_j\cdot x_k$ can then be calculated using the multiplication gadget shown in Figure~\ref{fig:multiplygadget}. It is simply a copy gadget upside down
that allows us to merge $1/x_j$ and $x_k$ into a common scale.

\begin{figure}[htb]\centering
 \includegraphics{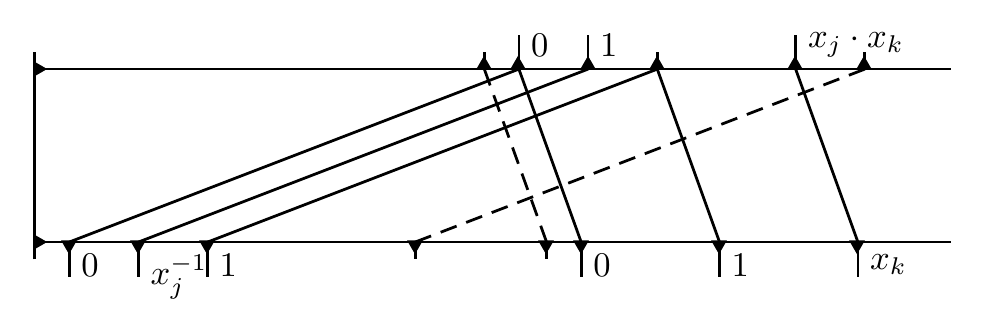}

 \caption{Multiplication gadget computing the product $x_i = x_j\cdot x_k$.}\label{fig:multiplygadget}
\end{figure}

The points along the top line (ignoring the endpoints of dashed lines which enforce orthogonality) represent $0$, $1/x_j$, $1$ and $x_k$, which we can relabel
as $0$, $1$, $x_j$ and $x_j\cdot x_k$, at which point we can drop the $x_j$ to
obtain a gadget computing $x_j\cdot x_k$. As did the reciprocal gadget above, this gadget changes the scale, and we need to rescale to
reestablish the standard distance between $0$ and $1$.

\subsection{Proof of Theorem~\ref{thm:TPvxNPR}}

We can assume that we are given a system of equations over variables $x_i$, $i \in [n]$ as described in Theorem~\ref{thm:ETReqs}. Our goal is to (efficiently) construct a graph $G$ with junctions, so
that the system of equations is solvable, if and only if the graph $G$ has a \RAC-drawing. Moreover, if the system is solvable, then $G$ has a crossing-free \RAC-drawing  (remember that junctions do not count as crossings).

We build $G$ in several stages. We first create gadgets for the main operations.
\begin{itemize}
\itemi For every variable $x_i$, $i \in [n]$, we create a variable gadget.
\itemii We have equations of four types: $x_i = 2$, $x_i = x_j$, $x_i = x_j+x_k$, and $x_i = x_j\cdot x_k$, and we create a corresponding gadget for each ($x_i = 2$ can be built by combining an equality gadget with a $2$-gadget; the addition and multiplication gadgets include the gadgets for negation and reciprocals).
\itemiii If variable $x_i$ occurs in $\ell$ equations, we create $\ell$ copy-gadgets.
\end{itemize}

Let use say this gives us $m$ gadgets. Place all $m$ gadgets along a vertical line spaced far apart, and so that incoming and outgoing (information carrying) edges are vertical.

We now need to connect the gadgets using angling gadgets. We need the following connections: the variable $x_i$ gadget needs to be connected as an input to one of the $\ell$ copy gadgets representing it, and we need to connect these copy gadgets so that we have $\ell$ outputs corresponding to $x_i$; finally, if $x_i$ occurs in an equation, we need to connect one of the unused outputs from the bank of copy gadgets representing $x_i$ to the equation. 

The connections between the $m$ original gadgets can be built from a chained sequence of at most eight angling gadgets.  This is already sufficient to get a $G$ that has a \RAC-drawing, but we need one more step to remove all crossings, so we describe how to place the angling gadgets more carefully; the placement only depends on the original equations, not on a specific solution to the equations.

Suppose we are connecting gadget $\alpha$ to gadget $\beta$,
with $\alpha, \beta < m$. Start at gadget $\alpha$. At most two angling gadgets let us leave the gadget $\alpha$ horizontally to the right up to a distance of $\alpha+\beta\cdot m$. Another two angling gadgets allow us to angle so we can move vertically downwards or upwards until we have reached the height of gadget $\beta$. Two more angling gadgets take us back horizontally to $\beta$, and we can connect to $\beta$ with another two angling gadgets.

This process introduces (orthogonal) crossings between edges of angling gadgets, but we can determine exactly what those crossings are (whether the gadgets are realizable, or not). We replace all crossings within gadgets and the newly introduced crossings between angling gadgets with \Pvx-junctions. The resulting graph is graph $G$ and we have also described a plane embedding of $G$ (which need not satisfy the junction-constraints, of course).

If the system of equations is solvable, we can use a solution to create a \RAC-drawing of each gadget, and, following the description above, create an isomorphic \RAC-drawing of $G$. In particular, there are no crossings.

On the other hand, if there is a \RAC-drawing of $G$, all the gadgets work as described, and each occurrence of a variable represents the same value, so there is a solution to the system of equations.


\subsection{Forcing Empty Faces}

To model \Tvx- and \Pvx-junctions as normal vertices we make use of another restricted
drawing mode that simplifies the construction. Given a graph $G$ with $k$ pairwise disjoint
sets of vertices $V_i \subseteq V(G)$, $i \in [k]$, we are interested in drawings (or \RAC-drawings) of $G$ in which the vertices of each $V_i$ lie on the boundary of an empty face, for $i \in [k]$.

The following theorem shows that this type of drawing constraint can be removed, even in \RAC-drawings.

\begin{theorem}\label{thm:emptyfaces}
  Let $G$ be a graph, and let $(V_i)_{i\in[k]}$ be $k$ pairwise disjoint sets of vertices of $G$. We can then construct, in polynomial time, a graph $G'$ so that $G$ has a \RAC-drawing in which the vertices of each $V_i$ lie on the boundary of an empty face, if and only if $G'$ has a \RAC-drawing.
\end{theorem}

For the proof, remember that if a graph has a \RAC-drawing, it can have at most $4n-10$ edges~\cite{DEL11}.

\begin{proof}
Let $v_1, \ldots, v_k$ be $k$ vertices not in $G$. Connect $v_i$ to each vertex in $V_i$ by a path of length $36n$, where $n = |V(G)|$, for $1 \leq i \leq k$. Finally, replace each of the newly added edges by $\ell = 146n^2$ paths of length $2$ (that is, a $K_{2,\ell}$). See Figure~\ref{fig:boundarypaths} for an illustration. Call the resulting graph $G'$. Suppose $G'$ has a \RAC-drawing.

\begin{figure}[htb]\centering
 \includegraphics[height=1.9in]{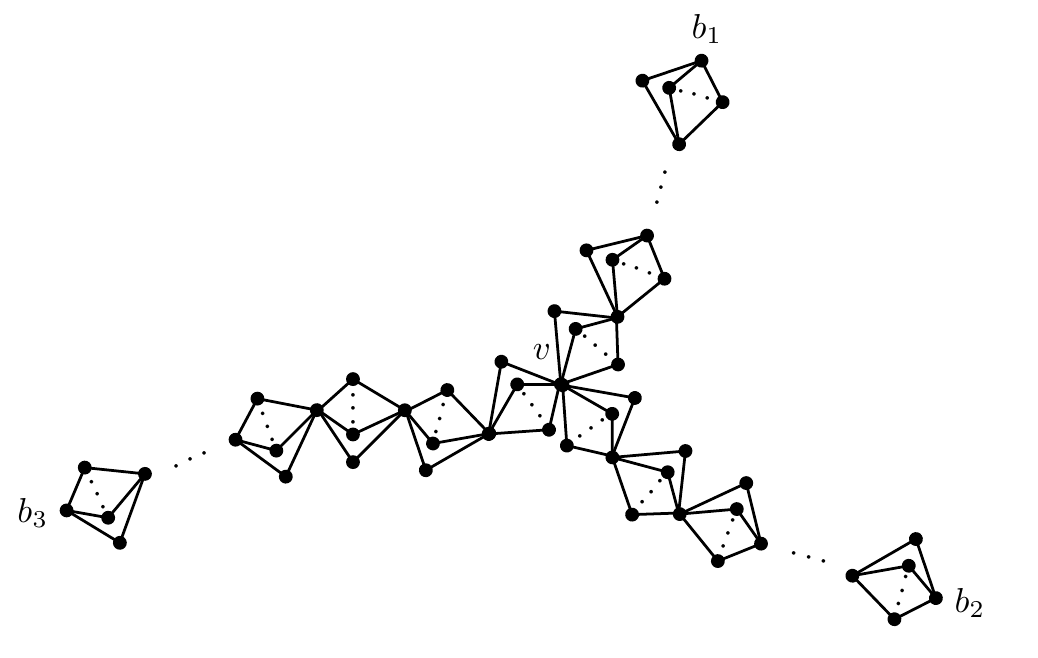}

 \caption{Forcing an empty face; in this example $v = v_i$ and $V_i = \{b_1, b_2, b_3\}$.}\label{fig:boundarypaths}
\end{figure}

Since we added at most $n$ paths of length $36n$ to $G$, we have at most $36n^2$ of the $K_{2,\ell}$-graphs in $G'$. We consider them one at a time. Let $E_0 = E(G)$. An edge can cross at most one edge of a $K_{1,\ell}$ at right angles (otherwise, the edges of $K_{1,\ell}$ would overlap). Hence, an edge can cross at most two edges of a $K_{2,\ell}$.
Since $\ell > 2n = 2|E_0|$, the first $K_{2,\ell}$ contains a path $P_1$ of length $2$ that crosses none of the edges in $E_0$. Let $E_1 = E_0 \cup E(P_1)$. If we keep repeating this argument, we obtain $E_i = E_{i-1} \cup E(P_i)$, and $|E_i| = n + 2i$. Then the $(i+1)$-st $K_{2,\ell}$ must contain a path $P_{i+1}$ of length $2$ which crosses none of the edges in $E_i$, since $2|E_i| = 2(n+2i) < \ell$, for all $1 \leq i \leq 36n^2$. We conclude that the \RAC-drawing of $G'$ contains a \RAC-drawing of $G$ together with the $v_i$ and paths from $v_i$ to each vertex in $V_i$ so that none of the paths are involved in any crossings. In other words, for each $i$ there is a crossing-free (subdivided) wheel with center $v_i$ and a perimeter containing $V_i$. Removing all edges not belonging to $G$ gives us a \RAC-drawing of $G$ in which all vertices of $V_i$ lie on the boundary of the same face (the one that contained $v_i$).

For the other direction, suppose $G$ has a \RAC-drawing in which all vertices of each $V_i$ lie on the boundary of the same face. For each $i$, we create a new vertex $v_i$ and a paths of length $72n$ connecting $v_i$ to each vertex in $V_i$. By Theorem~1 in~\cite{CFGLMS15}, the additional edges and vertices can be added to the already existing drawing of $G$ without creating any new crossings. We can then duplicate appropriate sub-paths of length $2$ to obtain a \RAC-drawing of $G'$.
\end{proof}

\section{Proof of Theorem~\ref{thm:RACNPR}}\label{sec:RACNPR}

Bieker~\cite[Section 6.2]{B20} shows that the problem lies in \NPR. To prove \NPR-hardness,
we are missing one more ingredient, a way to simulate junctions in \RAC-drawings.

\begin{theorem}\label{thm:remjunc}
 Let $D$ be a planar drawing of a graph $G$ with some vertices identified as \Tvx- and
 \Pvx-junctions. We can efficiently construct a graph $G'$ without junctions, with
  vertex sets $(V_i)_{i\in [k]}$, and a drawing $D'$ of $G'$ so that:
  \begin{itemize}
   \itemi if $D$ is isomorphic to a \RAC-drawing of $G$, then $D'$ is isomorphic to a \RAC-drawing of $G'$, and all edges are involved in at most ten crossings.
    \itemii if $G$ does not have a \RAC-drawing, then $G'$ does not have a \RAC-drawing
    in which all the vertices of each $V_i$ lie on a common boundary.
 \end{itemize}
\end{theorem}

The proof of Theorem~\ref{thm:remjunc} can be found in Section~\ref{sec:remjunc}.
Let us see how this theorem completes the proof of Theorem~\ref{thm:RACNPR}; we first consider the fixed embedding case. By Theorem~\ref{thm:TPvxNPR}
it is \NPR-hard to test whether a graph $G$ with junctions has a \RAC-drawing, even if we know
that the graph either has no \RAC-drawing, or that it has a \RAC-drawing isomorphic to a given planar drawing $D$. Using Theorem~\ref{thm:remjunc}, we construct a graph $G'$ and a drawing $D'$ so that
$D'$ is isomorphic to a \RAC-drawing if and only if $G$ has a \RAC-drawing. This implies that the fixed embedding version of \RAC-drawability is \NPR-complete.

To show that the problem is \NPR-complete without fixing the embedding, we need to take one
more step: By $(ii)$ we can reduce to $G'$ having a \RAC-drawing in which all the
vertices of each $V_i$ lie on a common boundary. Using Theorem~\ref{thm:emptyfaces}, this then reduces to \RAC-drawability (without a fixed embedding).

\section{Precision and Area}\label{sec:PaA}

Given an arbitrary algebraic number $\alpha$, we can write it as the unique solution of an integer polynomial equation (adding constraints to achieve uniqueness). E.g. $x = \sqrt{2}$ would be the unique $x$-value for which there is a solution of $(x^2 -2)^2 + (x-y^2)^2 = 0$. Using the reduction of polynomial equations to \RAC-drawability we have seen, we can build a graph $G$ so that in any realization of $G$ there are three collinear points that represent $\alpha$. We are not limited to just one number. For any semi-algebraic set $S$ in free variables $x_i, i \in [n]$ we can build a graph $G$ so that the triples $(0'_i,1'_i, x'_i)$ and $(0''_i,1''_i, x''_i)$  representing $x_i = x'_i-x''_i$ represent the semi-algebraic set in the sense that $S$ consists of the points $(d(0'_i,x'_i)/d(0'_i,1'_i)-d(0''_i,x''_i)/d(0''_i, 1''_i))_{i \in [n]}$.

Similarly, we can build a graph $G$ representing equations $x_1 = 2$, $x_2 = x_1\cdot x_1$, $x_3 = x_2\cdot x^2$, $\ldots$, $x_n = x_{n-1}\cdot x_{n-1}$. Then any realization of $G$ contains three points representing $0$, $1$, and $2^{2^n}$. So $G$ is a graph of polynomial size which requires
double-exponential area. In this example, the points in the gadgets can even be placed so as to lie
on a grid.

\section{Open Questions}

Can we relax the right-angle restriction? Huang, Eades, and Hong~\cite{HEH14} studied the impact of large angles (vs right angles) on the readability of drawings, and concluded (among other things) that large angles improve the readability of drawings, but the angles do not have to be right angles. They express the hope that the computational problem becomes easier if the right-angle restriction is relaxed. Using common tricks for the existential theory of the reals, it is possible to show that the \RAC-drawability problem remains \NPR-complete even if we relax the angle constraint, and require the angle to lie in the interval $(\pi/2-\varepsilon, \pi/2+\varepsilon)$, where $\varepsilon$ depends on $n = |V(G)|$ (doubly exponentially so). Does the problem remain \NPR-hard for a fixed value $\varepsilon>0$? The current construction very much relies on precision to simulate the existential theory of the reals. It is not clear whether gadgets can be braced to still work if angles are only approximate.

We saw that testing \RAC-drawability remains \NPR-hard, even if there is a \RAC-drawing with at most $10$ crossings per edge. Can that number be lowered? We note that the \NP-hardness result remains
true even for $1$-planar drawings (at most one crossing per edge)~\cite{BDL17}. Is
the problem in \NP? A somewhat similar situation occurs for the geometric local crossing number, $\rlcr(G)$, that is the smallest number of crossings along each edge in a straight-line drawing of $G$. Testing whether $\rlcr(G) \leq 1$ is \NP-complete~\cite{S14}, but there is a fixed $k$ so that
testing $\rlcr(G) \leq k$ is \NPR-complete~\cite{S21}.

Does \RAC-drawability remain \NPR-hard for bounded-degree graphs? Nearly all of our gadgets have bounded degree, the only exception are the empty-face gadgets, which are based on $K_{2,n}$'s, and
require unbounded degree. Can these be replaced with bounded degree gadgets?

The {\em right-angle crossing number} of a graph is the smallest $k$ so that $G$ has a \RAC-drawing with at most $k$ crossings. Our result implies that testing whether the right-angle crossing number is finite is \NPR-complete. What about small fixed values? Can we test whether a graph has a \RAC-drawing with one, two, three crossings in polynomial time? What about fixed $k$?

\bibliographystyle{splncs04}
\bibliography{rac-arxiv}

\appendix

\section{Proof of Theorem~\ref{thm:ETReqs}}\label{sec:ETReqs}

  We know that testing whether a given integer polynomial $f$ has a zero is \NPR-complete (see, for example~\cite[Corollary 4.2]{SS17}). Multiplying by $2$ we can ensure that all coefficients of $f$ are even numbers. We then replace every variable $x$ of $f$ with the difference of two new variables $x'-x''$. Then $f$ has a zero if and only if the new polynomial has a zero in which all variables are greater than $1$ (since any $x$ can be written
  as the difference of two numbers greater than $1$). We move negative terms to the other side of the equation to obtain two polynomials $g$, $h$ with positive, even coefficients so that $f$ has a zero if and only if $g = h$ has a solution with all variables being greater than $1$.

  We now calculate the values of $g$ and $h$ from the original variables step by step using new intermediary variables and only using equalities of the type $x_i = 2$, $x_i = x_j+x_k$ and $x_i = x_j \cdot x_k$. This includes the coefficients, which can be build from $2$ (since they are even).
  All the new intermediate variables are greater than $1$ (as products and sums of values greater than $1$). Finally, to test whether $g=h$ we add one more equation $x_i = x_j$, where $x_i$ is the new variable that computes $g$ and $x_j$ the variable that computes $h$. This completes the proof.

As an example, consider $f(x_1,x_2) = 3x_1^2 - x_2$. We first replace $f$ with $2f$, giving us $6x_1^2-2x_2$ with all even coefficients. We then replace each of the variables with a
difference of two new variables; in this example (reusing variable names) we get
$6(x_1-x_2)^2-2(x_3-x_4)$. Collecting positive terms gives us $6x_1^2+6x_2^2+2x_4 = 12x_1x_2+2x_3$,
so $g(x_1,x_2,x_3,x_4) = 6x_1^2+6x_2^2+2x_4$ and $h(x_1,x_2,x_3,x_4) = 12x_1x_2+2x_3$ in this example.
We then compute $g$ and $h$ term by term, using additional variables. E.g., let us show how to compute the term $6x_1^2$: we add equations $x_5 = x_1 \cdot x_1$, $x_6 = 2$, $x_7 = x_6 \cdot x_6$, $x_8 = x_6+x_8$, $x_9 = x_8 \cdot x_5$; then $x_9$ computes
$6x_1^2$. Similarly, we can compute all other terms into new variables, and add them up, one at a time, to get $g$ and $h$. Finally, we need one more equality, $x_i = x_j$, to compare the resulting values.

\section{Proof of Theorem~\ref{thm:remjunc}}\label{sec:remjunc}

We need gadgets to enforce the \Tvx- and \Pvx-junction restrictions. We will base these gadgets on the graph shown in Figure~\ref{fig:doublecap}.

\begin{figure}[htb]\centering
 \includegraphics{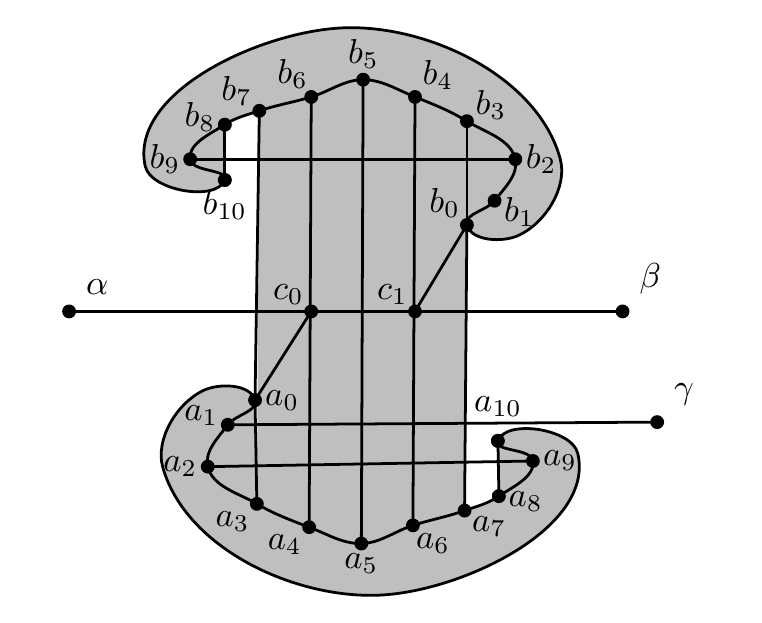}
 \caption{The double cap with attachments $c_0\alpha$, $c_1\beta$, and $a_0\gamma$. Each of the vertex sets $\{a_0,a_1,\ldots,a_{10}\}$
and $\{b_0, b_1, \ldots, b_{10}\}$ lies on the boundary of a face.
The inside of the double-cap is shaded.}\label{fig:doublecap}
\end{figure}

The double-cap simulates a line: $\alpha c_0$ and $c_1\beta$ will be collinear, and
$a_0\gamma$ will be parallel to that line; this allows us to enforce orthogonality.

Assuming the double-cap (without attachments) is drawn as shown in the illustration, we say that
$\alpha$, $\beta$, and $\gamma$ lie {\em outside the double-cap} if
they do not lie inside the region bounded by $a_0b_7b_8b_9b_{10}b_0a_7a_8a_9a_{10}a_0$.
This region is shaded in the figure.

\begin{lemma}\label{lem:doublecap}
Suppose we have a \RAC-drawing of the double cap, in which both $\{a_0,a_1,\ldots,a_{10}\}$
and $\{b_0, b_1, \ldots, b_{10}\}$ lie on the boundary of a face.
\begin{itemize}
\itemi
 Without the attachments, the drawing is isomorphic to the drawing shown in Figure~\ref{fig:doublecap} (without the attachments).
\itemii
  If $\alpha$, $\beta$, and $\gamma$ lie outside the double-cap, then the drawing of the double cap
with attachments is isomorphic to the one shown in Figure~\ref{fig:doublecap}.
\end{itemize}
\end{lemma}
\begin{proof}
  We start with $(i)$ and ignore the attachments to $\alpha$, $\beta$, and $\gamma$.
  Since $\{a_0, a_1, \ldots, a_{10}\}$ lie on the boundary of a face, and the ends of
$a_2a_9$ alternate with the ends of both $a_0a_3$ and $a_8a_{10}$, $a_2a_9$ must cross
both of those edges (orthogonally). Then $b_0c_0$ cannot cross $a_2a_9$, since otherwise
it would overlap with $a_0a_3$. It follows that $c_0a_4$ must cross $a_2a_9$. This implies
that $c_0c_1$ cannot cross $a_2a_9$ (edges would overlap), so $c_1a_5$ must cross $a_2a_9$.
Repeating the same argument, $c_1b_0$ cannot cross $a_2a_9$, so $b_0a_8$ must cross
$a_2a_9$. Similarly, $c_0b_6$ and $c_1b_4$ cannot cross $a_2a_9$, and since $b_4b_5b_6$ lie on
the boundary of the same face, $a_5b_5$ must cross $a_2a_9$. At this point, we know
that $a_0a_3$, $c_0a_4$, $b_5a_5$, $c_1a_6$, and $b_0a_7$ are all orthogonal to $a_2a_9$
and therefore all parallel to each other.

By symmetry, we can conclude analogous facts about the upper cap. It then follows that
$a_5b_5$ must cross $c_0c_1$, so that edge is parallel to $a_2a_9$ and $b_2b_9$. And
edges $a_0a_3$ and $a_0b_7$ are collinear, as are $c_0a_4$ and $c_0b_6$, as well
as $c_1a_6$ and $c_1b_4$, and $b_0a_7$ and $b_0b_3$.

This, in turn implies that each of $a_3a_0b_7$, $a_4c_0b_6$, $a_5b_5$, and $a_6c_1b_4$
must cross edge $b_0b_{10}$ (since they connect the two caps, both of which lie on face boundaries).

In summary, without attachments,
the drawing of the double cap is (up to a homeomorphism) as shown in Figure~\ref{fig:doublecap}.

This allows us to speak of $\alpha$, $\beta$ and $\gamma$ lying outside the double-cap.
Since $\gamma$ does not lie inside the region bounded
by $a_0a_3$ and the boundary of $\{a_0,a_1,\ldots,a_{10}\}$, $a_1\gamma$ must cross
$a_0a_3$ (since the boundary is free of crossings), so it enters the rectangle formed
by $b_7a_0a_3$, $a_2a_9$, $a_7b_0b_3$, and $b_2b_9$. Since it does not lie inside,
it must continue crossing $c_0a_4$, $b_5a_5$, $c_1a_6$ and $b_0a_7$ (note that $a_0c_0$
cannot cross $a_1\gamma$, so $c_1c_2$ lies above $a_1\gamma$. Hence, $a_0\gamma$ is
as shown in the drawing.

Next, consider $\alpha$. Since it does not lie inside the rectangle formed
by $b_7a_0a_3$, $a_2a_9$, $a_7b_0b_3$, and $b_2b_9$, it must be that $\alpha c_0$ crosses one of
the four sides of that recangle, but three of the directions, $c_0b_6$, $c_0c_1$, and $c_0b_4$ are
already taken (since there can be no overlap). So $\alpha c_0$ must cross either
$a_0a_3$ or $a_0b_7$. Since it cannot cross the boundary of the $a$-cap, it must cross $a_0b_7$.

An analogous argument for $\beta$ show that $c_1\beta$ crosses $b_0a_7$. Hence, the
drawing with attachments is as shown in the figure.
\end{proof}

By combining two double-caps we can now build a \Tvx-gadget, as shown in Figure~\ref{fig:Tjunctiongadget}.
There are four attachments: $v\beta$, $c_0\alpha$, $c_0'\gamma$, and $v\delta$. (We label each double cap as before, and use $'$ to
refer to vertices in the right double cap.

We define the {\em outside region} of a \Tvx-gadget to be everything that is outside
both the double-caps (as defined earlier) as well as outside the region bounded by
the boundary $c_0a_0a_1a'_1c'_0c_0$. The inside region is shaded in the drawing.

\begin{figure}[htb]\centering
 \includegraphics[width=5in]{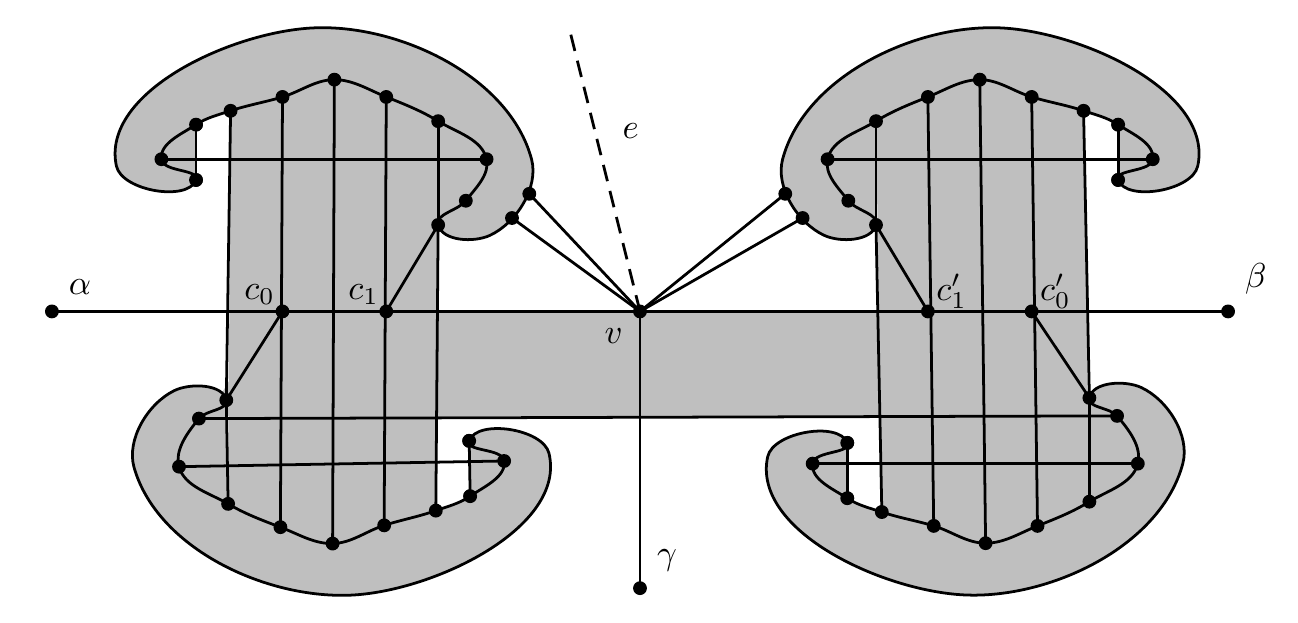}
 \caption{Simulating a \Tvx-junction with two double-caps. The inside
of the region is shaded.}\label{fig:Tjunctiongadget}
\end{figure}

\begin{lemma}\label{lem:Tjunctiongadget}
 In a \RAC-drawing of the \Tvx-junction gadget, the drawing, without attachments, is
as shown in Figure~\ref{fig:Tjunctiongadget} (without the attachments). If $\alpha$,
$\beta$ and $\gamma$ lie outside the \Tvx-junction, then the drawing
is as shown in the figure.
\end{lemma}

Edge $e$ could leave $v$ in other directions as well, the drawing shows the intended
direction. We note that the caps can be made arbitrarily small, and close to $v$, so that $e$
can be realized anywhere in the rotation at $v$ between $\alpha$ and $\beta$.

\begin{proof}
By Lemma~\ref{lem:doublecap}, we can assume that the left double cap, without attachments, is drawn as
shown in the figure. Similarly, the right double cap without attachments is either drawn as shown, or
reversed (we will exclude that possibility later).

We attached $v$ to two new vertices on the boundary of one of the left $b$-cap. Then
$v$ must lie outside the left double-cap: to lie inside, it the edges of attachment would
have to cross $b_8b_{10}$, $a_8a_{10}$, $b_0a_7$ or $a_7b_7$. Two incident edges cannot
both cross the same edge orthogonally, so the two edge incident to $v$ must cross two
different edges. This is only possible if $v$ lies inside the region bounded by
the double-cap and the two attachments cross $a_0b_7$ and $b_0a_7$. This is not
possible, because of the two diagonal edges $c_0a_0$ and $c_1b_1$. We conclude
that $v$ lies outside of the left double-cap, and, with the same argument, outside the
other double-cap as well. Moreover, the two double caps are outside each other, since otherwise
$vc_1$ would overlap with $vc'_1$. At this point, we know that the right double cap
is orient as shown in the figure, and not reversed.

Consider the edge $a_0a'_0$. Since $a'_0$ is outside the left double cap, $a_0a'_0$ crosses
$b_0a_7$ at a right angle, and, since $a_0$ is outside the right double cap, $a_0a'_0$
crosses $b'_0a'_7$ at a right angle. Hence, if $\alpha$, $\beta$, and $\gamma$ lie
outside the \Tvx-junction, $\alpha v$ and $v\beta$ lie on a line, and $v\beta$ is orthogonal
to that line (since it crosses $a_0a'_0$).
\end{proof}

Finally, as we mentioned earlier, we can replace each \Pvx-junction with four \Tvx-junctions, see Figure~\ref{fig:PvxtoTvx}.

\begin{figure}[htb]\centering
 \includegraphics[height=1.5in]{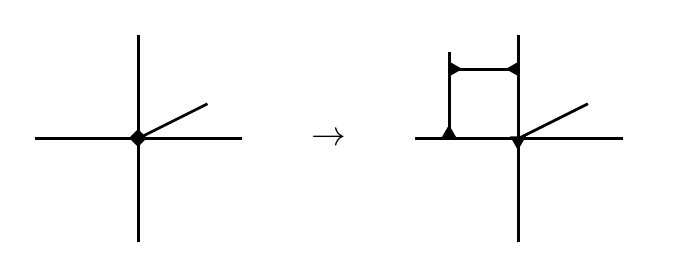}
 \caption{How to replace a \Pvx-junction with four \Tvx-junctions.}\label{fig:PvxtoTvx}
\end{figure}

\begin{lemma}\label{lem:PvxtoTvx}
 Given a graph $G$ with \Pvx- and \Tvx-junctions, we can build a graph $G'$ with \Tvx-junctions only so that $G$ has a \RAC-drawing if and only if $G'$ has, and if $G$ has a crossing-free \RAC-drawing, then so does $G'$. Given a drawing $D$ isomorphic to a \RAC-drawing of $G$ we can efficiently find a drawing $D'$ isomorphic to a \RAC-drawing of $G'$.
\end{lemma}
\begin{proof}
  We replace each \Pvx-junction with four \Tvx-junctions as shown in Figure~\ref{fig:PvxtoTvx}. This does not affect \RAC-drawability, and constructs $D'$ efficiently from $D$. If the original drawing is crossing free, then, the \Tvx-junctions can be placed arbitrarily close to the vertex of the original junction so that
  the new \RAC-drawing remains crossing-free.
\end{proof}

Because of the lemma, we only need to deal with \Tvx-junctions. The remaining issue is that we need to keep multiple \Tvx-junction gadgets from overlapping or interfering with each other. Here is where the planarity of the graph will play a role.

Suppose we are given a graph $G$ with \Tvx-junctions (which we can assume by Lemma~\ref{lem:PvxtoTvx}), and a crossing-free drawing $D$ of $G$. We know that a \RAC-drawing of $G$, if it exists, will be isomorphic to $D$. 
We start  by replacing each \Tvx-junction with a \Tvx-junction gadget. 

Consider an edge $uv$ in $G$. If it connects
two non-junction vertices, we do nothing. Suppose $uv$ connects \Tvx-junction $v$ to a vertex $u$ in $G$.
If $u$ is $\alpha$ or $\beta$ in the \Tvx-junction gadget, we connect $u$ with two edges to one of the neighboring two caps belonging to the \Tvx-junction gadget belonging to $v$. If $u$ is $\gamma$ in the \Tvx-junction gadget, we connect it with two edges to the lower caps of both double-caps belonging to $v$.
See the vertex $u$ on the left of Figure~\ref{fig:joinjunctions}.

\begin{figure}[htb]\centering
 \includegraphics[width=5in]{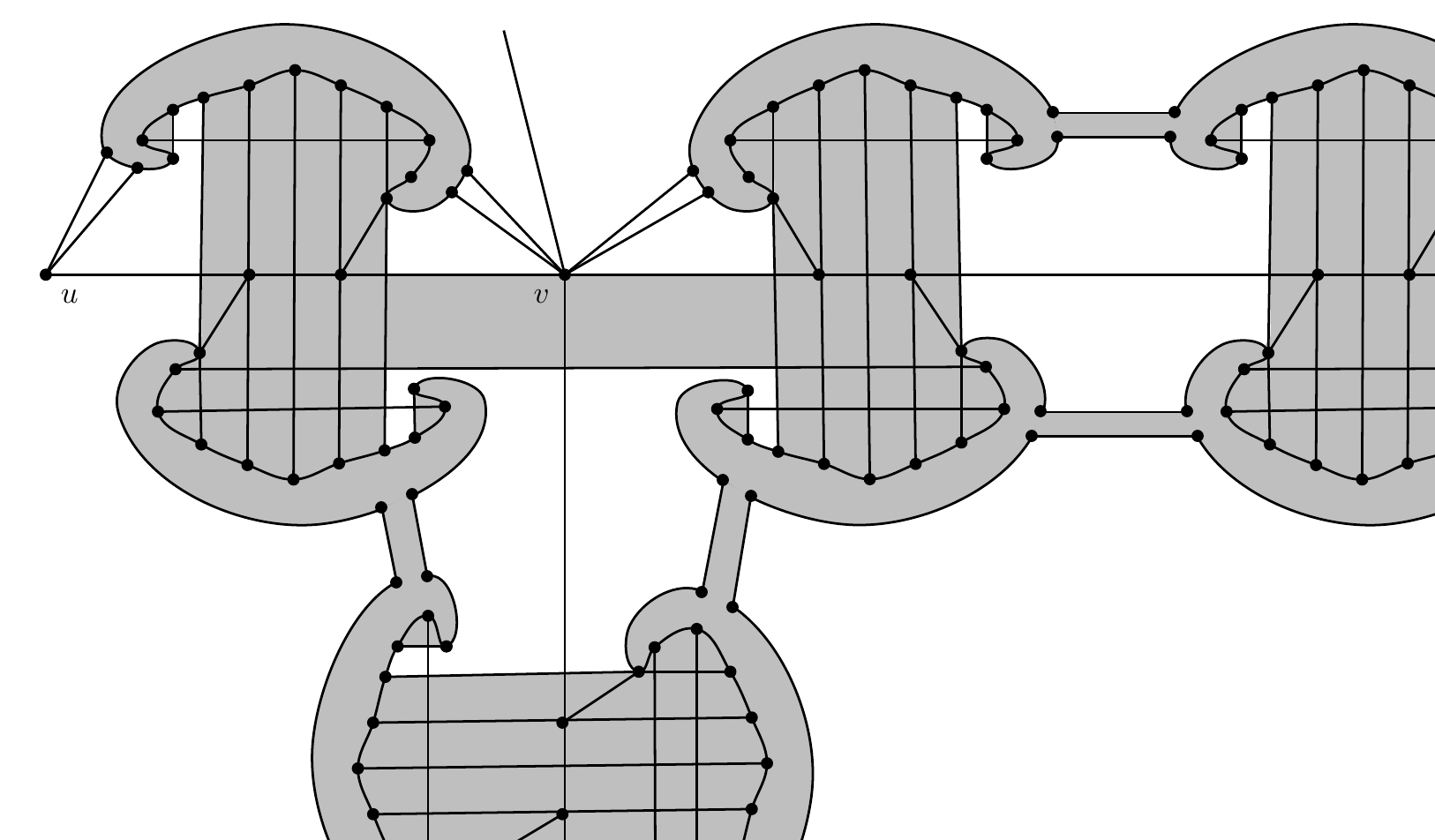}
 \caption{Connecting a \Tvx-junction to neighboring \Tvx-junctions and vertices.}\label{fig:joinjunctions}
\end{figure}

Similarly, if an edge $vw$ connects two \Tvx-junctions in $G$, we connect the boundaries of all caps (of \Tvx-junction gadgets belonging to $u$ and $v$) that lie on the same face and merge them. See Figure~\ref{fig:joinjunctions} which shows the caps of a \Tvx-junction gadget merged along a $\beta$- and a $\gamma$-edge. This forces the $c_1$ vertex belonging to $w$ to lie outside the \Tvx-junction gadget belonging to $v$ and vice versa. In this fashion, we ensure that for each \Tvx-junction gadget, its $\alpha$-, $\beta$-, and $\gamma$-vertices lie outside the gadget, which, by Lemma~\ref{lem:Tjunctiongadget} implies that the gadget correctly represents a \Tvx-junction.

From the plane drawing $D$ of $G$ we have constructed a graph $G'$ (without \Tvx- or \Pvx-junctions) and a plane drawing $D'$ of $G'$. Let $(V_i)_{i\in [k]}$ be the collection of boundaries of (merged) caps.

If $D$ is isomorphic to a \RAC-drawing of $G$, then we can draw the junction gadgets as intended, and $D'$ is isomorphic to a \RAC-drawing of $G'$. Inspecting the gadgets, we find that no edge in any gadget is involved in more than $10$ crossings (this number is achieved for edges of type $a_0a'_0$ connecting two double-caps in a \Tvx-junction gadget). Hence, we can assume that the \RAC-drawing is $10$-planar. This proves property $(i)$ of the theorem. 

For $(ii)$ suppose that $G'$ has a drawing in which the vertices of each $V_i$ lie on a face. This is sufficient, as we argued, for all the junction gadgets to work correctly, and we can conclude that $G$ has a \RAC-drawing. This completes the proof of the theorem.
\end{document}